\documentclass[a4paper,10pt]{article}
\usepackage[]{inputenc}
\usepackage{amsmath, amsthm,amssymb}
\numberwithin{equation}{section}
\usepackage{enumerate}
\usepackage[bottom=2.2cm, top=2.2cm,inner=2.7cm,outer=3.3cm]{geometry}

\newtheorem{Theorem}{Theorem}[section]
\newtheorem{Proposition}[Theorem]{Proposition}
\newtheorem{Lemma}[Theorem]{Lemma}
\newtheorem{Corollary}[Theorem]{Corollary}

\theoremstyle{definition}
\newtheorem{Definition}[Theorem]{Definition}

\newtheorem{Remark}[Theorem]{Remark}
\newtheorem{Fact}[Theorem]{Fact}

\newtheorem{Convention}[Theorem]{Convention}

\newcommand{\rvu}[0]{\mbox{\textup{RV}}}
\newcommand{\rve}[0]{\mbox{\textup{RV}$^{\textup{\small eq}}$} }
\newcommand{\vf}[0]{\mbox{\textup{VF}}}

\newcommand{\gl}[1]{GL$_{#1}(O_R)$}
\newcommand{\gln}[0]{\gl{n}}

\newcommand{\lto}[0]{\longrightarrow}

\newcommand{\tsp}[0]{\text{tsp}}
\newcommand{\rcf}[0]{\text{RCF}}

\newcommand{\lcon}[0]{L_{\text{convex}}}
\newcommand{\tcon}[0]{T_{\text{convex}}}
\newcommand{\lori}[0]{L_{\text{or}}}

\newcommand{\lrv}[0]{L_{\scriptsize \rvu}}
\newcommand{\trv}[0]{T_{\scriptsize \rvu}}
\newcommand{\lhen}[0]{L_{\text{Hen}}}
\newcommand{\then}[0]{T_{\text{Hen}}}

\newcommand{\starr}[0]{{}^*\mathbb R}

\newcommand{\sat}[0]{(R^*,\rvu^*)}

\newcommand{\subrv}[0]{{_{\textup{RV}}}}
\newcommand{\gen}[1]{\langle {#1}\rangle}
\newcommand{\genl}[1]{\langle {#1}\rangle_{L}}

\DeclareMathOperator{\res}{res}
\DeclareMathOperator{\rvo}{rv}

\DeclareMathOperator{\graph}{graph}
\DeclareMathOperator{\afd}{affdir}
\DeclareMathOperator{\jac}{Jac}

\DeclareMathOperator{\val}{v}

\DeclareMathOperator{\vrv}{\val_{_{\textup{RV}}}}
\DeclareMathOperator{\rad}{rad}

\DeclareMathOperator{\acl}{acl}
\DeclareMathOperator{\dcl}{dcl}

\DeclareMathOperator{\tp}{tp}

\usepackage{xcolor}

\title{\Large Non-archimedean stratifications in power bounded $T$-convex fields}
\author{{\normalsize Erick Garc\'ia Ram\'irez}\\
				{\small \textit{University of Leeds, United Kingdom}}}
\date{\small \today}
\begin{document}
\maketitle
\begin{abstract}
We show that functions definable in power bounded $T$-convex fields have the (multidimensional) Jacobian property. Building on work of I. Halupczok, this implies that a certain notion of non-archimedean stratifications is available in such valued fields. From the existence of these stratifications, we derive some applications in an archimedean o-minimal setting. As a minor result, we also show that if $T$ is power bounded, the theory of $T$-convex valued fields is $b$-minimal with centres.
\end{abstract}
\textbf{Keywords:} $T$-convex fields, t-stratifications, Jacobian property.
{\let\thefootnote\relax\footnotetext{\textit{Email:} mmegr@leeds.ac.uk}}

\section{Introduction}\label{introduction}
The present work has two interrelated aims in view. The first is to advance the study on definability in $T$-convex fields; in this respect this paper is in line with the work of L. van Dries and A. Lewenberg~\cite{driesII}~\cite{driesI}, the work of Y. Yin~\cite{yin} and, in a slightly different character, with the work of R. Cluckers and F. Loeser on $b$-minimality~\cite{raf2}. The second aim is to make the non-archimedean stratifications introduced by I. Halupczok in~\cite{immi}  available in the context of power bounded $T$-convex fields.  The latter aim touches on the study of singularities of definable sets.

Let $L$ be a language containing the language $\lori:=\{+,\cdot, 0,1,<\}$ of ordered rings, and $T$ a power bounded o-minimal $L$-theory extending $\rcf$, the $\lori$-theory of real closed fields. Assume that $R$ is a model of $T$. A $T$-\emph{convex subring} of $R$ is a proper convex subring $O_R\subseteq R$ closed under all $0$-definable continuous functions $f:R\lto R$. Such a subring is a valuation ring and thus we think of the pair $(R,O_R)$ as a valued field. If $M_R$ is the maximal ideal of $O_R$, we define $\rvu:=R^\times/(1+M_R) \cup \{0\}$. $\rvu\setminus \{0\}$ has a natural ordered group structure that can be extended conventionally to the whole of $\rvu$. We work with the two-sorted structure $(R,\rvu)$, where $R$ is seen as an $L$-structure, $\rvu$ as a $\{\ast, 1,0,<\}$-structure, and we connect the two sorts via the canonical map $\rvo:R\lto \rvu$ (extended by $\rvo(0)=0_{_{\textup{RV}}}$). Although we mostly work in this two-sorted language, our main results make reference to a larger language consisting of the sort $R$ as an $L$-structure, plus all the imaginary sorts coming from $\rvu$, along with all the natural maps between these sorts and these sorts and $R$. We write $(R, \rve)$ whenever we subscribe ourselves to the latter approach. 

 The attached valuation map on $R$ is denoted by $\val$ and is extended to $R^n$ by $\val((x_1,\dots,x_n)):=\min \{\val(x_i)\ | \ 1\leq i\leq n\}$. Our first main result is the following.

\begin{Theorem}\label{main1}
Any definable function in $(R,\rve)$ has the Jacobian property. This is, if $f:R^n\lto R$ is  definable in $(R,\rve)$, then there exists a map $\chi:R^n\lto S\subseteq \rve$ definable in $(R,\rve)$ such that if $q\in S$ satisfies that $\dim(\chi^{-1}(q))=n$, then there exists $z\in R^n\setminus \{0\}$ such that \begin{equation*}
\val(f(x)-f(x')-\langle z,x-x'\rangle)>\val(z)+\val(x-x'),
\end{equation*}
for all distinct $x,x'\in \chi^{-1}(q)$.
\end{Theorem}

The main consequence of such result---and certainly the motivation for this research---is the following theorem, now a simple corollary of Theorem~\ref{main1} and the work of I. Halupczok in~\cite{immi}.

\begin{Theorem}\label{main3}
Any definable map $\chi:R^n\lto \rve$ admits a t-stratification. In particular, definable sets in $(R,\rve)$ admit t-stratifications.
\end{Theorem} 

These results depend heavily on the power boundedness of $T$. In Subsection~\ref{notstrats} we give an example of a non power bounded o-minimal theory and a definable set not admitting a t-stratification. The example exploits the behaviour of an exponential map. Since a model of a non power bounded theory $T$ necessarily defines an exponential map (this is the \emph{dichotomy} in o-minimal theories, see~\cite{miller2}), the example is prevalent for all such $T$.

In the process of proving Theorem~\ref{main1} we deal with plenty of the properties of $\lrv$-definable sets. Part of this analysis builds on the work of Y. Yin in~\cite{yin}. Some of such properties resemble those explored in~\cite{raf2}. In fact, well before reaching Theorem~\ref{main1}, we prove the following.

\begin{Theorem}\label{main2}
The theory of $T$-convex fields (in the several-sorted language described above) is $b$-minimal with centres (over the sorts $\rve$).
\end{Theorem}  

The last section of the paper starts by presenting an application foreseen in~\cite[Section 4]{egr}. It describes how t-stratifications of a set induce t-stratifications on tangent cones of that set. We then move onto prove that the t-stratifications obtained from Theorem~\ref{main3} can be refined by t-stratifications with $L$-definable strata. This is used in our second application, in which we work with stratifications in an archimedean setting. We regard the real field $\mathbb R$ as an $L$-structure. A non-principal ultrapower $^*\mathbb R$ of $\mathbb R$ can naturally be made into a $T$-convex field. We then look at the stratifications induced on $\mathbb R$ by the t-stratifications in $^*\mathbb R$. By guaranteeing that the results in~\cite[Section 7]{immi} hold in our context, we establish the following. 

\begin{Theorem}\label{main4}
Let $X\subseteq R^n$ be $L$-definable and $S_0,\dots ,S_n$ be an $L$-definable partition of $X$. If the sets ${}^{*}S_0,\dots ,{}^{*}S_n$ form a t-stratification of ${}^*X$, then the sets $S_0,\dots ,S_n$ form a  $C^1$-Whitney stratification of $X$. 
\end{Theorem}

When $L=\lori$, Theorem~\ref{main4} was proved in~\cite[Section 7.3]{immi}.

A minor third application, reachable by having ensured that t-stratifications \emph{can be made $L$-definable}, is a new proof of the existence of ($C^1$-) Whitney stratifications for $L$-definable subsets of $\mathbb R^n$. This result is well-known and in fact does not require the assumption of power boundedness on $T$, see~\cite{loi}.
\newline

{\bf \noindent Acknowledgements.} This work was carried out as part of the author's PhD programme. The author wishes to thank Prof Immanuel Halupczok and Prof Dugald Macpherson for their supervision and invaluable support. Financially, the author was supported by CONACYT and DGRI-SEP. 

\section{Definability in power bounded $T$-convex fields}\label{definability}
In this section we collect several results on definability. Most of these results are technical and are fundamental to our later work.

We start by giving more details on the languages used in this paper to study $T$-convex fields. The first language comes from the foundational work in~\cite{driesII} and~\cite{driesI} and corresponds to adding to $L$ a unary predicate to be  interpreted as a $T$-convex subring. This language is denoted by $\lcon$, and the common $\lcon$-theory of pairs $(R,O)$, with $O\subseteq R$ a $T$-convex subring of $R$, will be denoted as $\tcon$. Models of $\tcon$ will be denoted as pairs $(R,O_R)$.

Keeping the notation above, the second language corresponds to adding a new sort for $\rvu:=R^\times/(1+M_R)\cup \{0\}$, where $M_R$ is the maximal ideal of $O_R$. On this new sort we put the language of ordered groups $L_{og}:=\{\ast, 1,0,<\}$. The operation $\ast$ has the obvious interpretation on $R^\times/(1+M_R)$ and is extended by putting $0_{\scriptsize \rvu}\ast \xi=0_{\scriptsize \rvu}$, for each $\xi\in R^\times/(1+M_R)$. The canonical map $\rvo:R\lto \rvu$, extended by $\rvo(0)=0_\subrv$, is the only symbol connecting the two sorts. The order on $\rvu$ is given such that $\xi<_\subrv \xi'$ if and only if $\rvo^{-1}(\xi)<\rvo^{-1}(\xi')$, for all $\xi,\xi'\in \rvu$. This two-sorted language will be denoted by $\lrv$. This approach is exploited in~\cite{yin}, and we will in fact refer to that paper frequently. The common $\lrv$-theory of $T$-convex fields is denoted by $\trv$. 

For the language $\lrv$, we may assume that on the sort $\rvu$ we have a predicate for the residue field $\overline R$ of $R$. Incidentally, the residue map from $O_R$ to $\overline R$ is denoted by $\res$ and it is extended to $O_R^n$ coordinate-wise. In addition, we think of the value group $\Gamma_{\infty}$ as another sort and we consider both the valuation map on $R$, $\val:R\lto \Gamma_{\infty}$ and the valuation map on $\rvu$, $\vrv:\rvu\lto \Gamma_{\infty}$ as part of the language; by the way, these two maps are naturally extended to $R^n$ (see the introduction). These assumptions are no more than convenient.

Finally, the third language enriches $\lrv$ by considering all imaginary sorts on the sort $\rvu$, this is, we let $\rve$ be the union of all quotients of $\rvu$ by 0-$\lrv$-definable equivalence relations. For every such quotient, we add the associated canonical map from $\rvu$ into it, and also any other natural map from $R$ into it. The  language will be denoted by $\lhen$.  For example, on $R^n$ we define the relation $x\sim y$ if and only if either $\val(x-x')>\val(x)$ or $x=x'=0$.  The quotient is denoted by $\rvu^{(n)}$ and the natural map $R^n\lto \rvu^{(n)}$ will be denoted ---under an innocuous abuse of notation--- by $\rvo$. All the quotients $\rvu^{(n)}$ are then part of $\rve$ and the map(s) $\rvo$ is added to the language. By $\then$ we denote the common $\lhen$-theory of $T$-convex fields. 

We remind the reader that throughout this paper we assume that $T$ is power bounded, with the sole exception of Subsection~\ref{notstrats}. We now cite a few results on $\lrv$-definability. For us, \emph{definable} means  definable with parameters. Nevertheless, we tend to be thoroughly precise on the parameters used. For instance, we will frequently write \emph{$X$ is $A$-$\lrv$-definable} to mean that $X$ is definable by an $\lrv$-formula with parameters strictly from $A$. 

\begin{Theorem}\label{qe}
The theory $\trv$ admits quantifier elimination. 
\end{Theorem}
\begin{proof}
 \cite[Theorem 1.8]{yin}; the proof reduces the result to the quantifier elimination of $\tcon$ and employs the \emph{Wilkie inequality}. The latter is only available when $T$ is power bounded, see~\cite[\S 5]{driesII}.
\end{proof}

It follows that each $\lrv$-formula $\phi(x)$, with $x$ a single field variable, is equivalent to a boolean combination of formulas of the form 
\begin{align}\label{eq:3.1}
&x=a,& & x<a,& &\rvo(x-a)=\rvo(b), & & \rvo(x-a)< \rvo(b),& &    
\end{align}
for some $a,b\in R$.
Therefore, each $\lrv$-definable subset of $R$ is a finite union of singletons, open intervals and v\emph{-discs}\index{v-disc}, where a v-disc is a set of the form $\{x\in R\ |\ \xi_1\square_1\rvo(x-a)\square_2 \xi_2\}$, with $a\in R$ and $\xi_i\in \rvu$, and $\square_i\in \{\leq,< \}$ ($i=1,2$). This property of $\lrv$-definable subsets of $R$ is only available when $T$ is power bounded, see~\cite[Observation 7.3]{driesII}.

Notice that in~\eqref{eq:3.1} the first two formulas are redundant: $x=a$ can be replaced by $\rvo(x-a)=0_\subrv$, and $x<a$ by $\rvo(x-a)<0_\subrv$. The similar observation can be made about multi-variate formulas. We therefore adhere to the following convention.
\begin{Convention}\label{rvterms}
	Let $x$ be a tuple of field variables. We will assume that each $L$-term $t(x)$ appearing in an $\lrv$-formula occurs always under the scope of an instance of $\rvo$. So each $\lrv$-formula with free field variables in $x$ has the form 
	\begin{equation*}
	\phi(\rvo(t_1( x, a)),\dots,\rvo(t_m( x, a)), z, \eta),
	\end{equation*} 
	where $\phi$ is a quantifier-free formula in the language of the sort $\rvu$, $z$ is a tuple of $\rvu$ variables, each $t_i$ is an $L$-term and $ a$ and $ \eta$ are tuples of elements in $R$ and $\rvu$, respectively.
\end{Convention}

 An interval is a set of the form $\{x\in R\ |\ a\square_1 x\square_2 b\}$, with $\square_i\in \{<,\leq\}$ for each $i=1,2$. We use standard notation for these sets. Note that singletons are considered as intervals. The following proposition is among the first results on $\lrv$-definability; it is a sharper version of~\cite[Corollary 2.8]{driesI}, and it was proved first by Y. Yin.

\begin{Proposition}\label{sharperwomono}
  Let $A$ be a subset of $R\cup \rvu$. Suppose that $C\subseteq R$ and that $f:C\lto R$ is an $A$-$\lrv$-definable. Then there is an $A$-$\lrv$-definable finite partition $C_1,\dots, C_m$ of $C$, with each $C_j$ either an interval or a $\val$-disc, and $L$-definable functions $f_1,\dots,f_m:R\lto R$ such that for each $j\in \{1,\dots,m\}$, $f|_{C_j}=f_{j}|_{C_j}$. Moreover, each function $f_j$ can be assumed to be continuous, constant or strictly monotone, and even, if desired, differentiable on $C_j$.
\end{Proposition}
\begin{proof}
\cite[Lemma 2.1 and Corollary 2.2]{yin}. 
\end{proof}

\begin{Remark}
1. Under the conclusion of the proposition above, for each $j\in \{1,\dots, m\}$, the continuity and monotonocity of $f$ on $C_j$ imply that $f(C_j)$ is either an interval or a v-disc.

2. The following multi-variate version of Proposition~\ref{sharperwomono} can be proved via induction and compactness: If $C\subseteq R^n$ and $f:C\lto R$  is an $A$-$\lrv$-definable function, then there is an $A$-$\lrv$-definable finite partition $C_1,\dots,C_m$ of $C$ and $L$-definable functions $f_1,\dots, f_m:R^n\lto R$ such that $f|_{C_j}=f_j|_{C_j}$ for each $j\in \{1,\dots ,m\}$. 
\end{Remark}

\begin{Corollary}\label{JP0lrv}
Let $A$ be a subset of $R\cup \rvu$. Suppose that  $X\subseteq R$ and that $f:X\lto R$ is an $A$-$\lrv$-definable function. Then there is an $A$-$\lrv$-definable map $\chi:R\lto \rvu$ such that for each $q\in \chi(R)$, $f|_{\chi^{-1}(q)}$ is either constant or injective. 
\end{Corollary}
\begin{proof}
 By Proposition~\ref{sharperwomono}, there is a finite $\lrv$-definable partition $C_1,\dots,C_m$ such that $f|_{C_j}$ is either constant or strictly monotone for each $1\leq j\leq m$. It is enough to put $\chi(x):=\rvo(j)\in \rvu$, whenever $x\in C_j$.
\end{proof}

The following proposition implies, among other consequences, that in both the field-sort and the $\rvu$-sort we have a working reasonable dimension theory. This fact is used freely later. 
\begin{Proposition} Let $(R,O_R)$ be a model of $\trv$. Then,
 \begin{enumerate}
  \item[\textup{(a)}] the exchange principle holds in the field-sort $R$, i.e. for any tuples $a,b$ of elements of $R$, $b\in \acl(a)\setminus \acl(\varnothing)$ implies that $a\in \acl(b)$;
  \item[\textup{(b)}] the exchange principle holds in the $\rvu$-sort.
 \end{enumerate}
\end{Proposition}
\begin{proof}
The proof of~\cite[Proposition]{yin2} works in this context.
\end{proof}

The next result is of high importance since we want to keep a tight control on the use of parameters. The result is originally~\cite[Proposition]{yin}.

\begin{Lemma}\label{lemma3}
 Given a tuple $\xi\in \rvu^n$, if $a\in R$ is $\xi$-$\lrv$-definable, then $a$ is $\varnothing$-$\lrv$-definable. 
\end{Lemma}
\begin{proof}
 The result follows for $n>1$ from the case $n=1$ via induction. For $n=1$, notice that if $\phi(x,\xi)$ is an $\lrv$-formula defining $a$ ($\neq 0$), then $a$ is also defined by the $\lrv$-formula $\phi(x,\rvo(b))$ for all $b\in \rvo^{-1}(\xi)$, so $a\in \acl(b)$ for each $b\in \rvo^{-1}(\xi)$. If $a\notin \acl(\varnothing)$, the exchange principle (in the field-sort) implies that $b\in \acl(a)$ for all $b\in \acl(a)$. Hence $\rvo^{-1}(\xi)\subseteq \acl(a)$. By compactness we deduce that $\rvo^{-1}(\xi)$ is a finite set, which is absurd. Thus indeed, $a\in \acl(\varnothing)$.
\end{proof}

\begin{Proposition}\label{finiteimage}
Suppose that $\Xi\subseteq \rvu$ and let $p:\Xi\lto R$ be an $\lrv$-definable function. Then $p(\Xi)\subseteq R$ is finite. 
\end{Proposition}
\begin{proof}
 Evidently, for each $\xi\in \Xi$, $p(\xi)$ is $\xi$-$\lrv$-definable. By Lemma~\ref{lemma3}, each element of $p(\Xi)$ is then $\varnothing$-$\lrv$-definable. The conclusion follows by compactness.
\end{proof}

\begin{Remark}
The conclusion of Proposition~\ref{finiteimage} remains true for any $\lrv$-definable function $p:\Xi\lto R^n$, where $\Xi\subseteq \rvu^m$, for all $n,m\geq 1$.
\end{Remark}

\begin{Proposition}\label{finiteinjection}
Let $A$ be a subset of $R\cup \rvu$. If $X\subseteq R$ is a finite $A$-$\lrv$-definable set, then there is an $A$-$\lrv$-definable injection $j:X\lto \rvu^k$, for some $k\geq 1$.
\end{Proposition}
\begin{proof} If $X$ is a singleton, the conclusion is trivial. Let $X=\{b_1,\dots b_m\}$. By considering  $b_i':=b_i-\frac{1}{m}\sum_{j=1}^{m}b_j$ we can assume that the average of $X$ is 0. We can further assume that $\val$ is constant on $X$ (otherwise, clearly $\rvo$ is not constant on $X$ and we can apply the upcoming argument). Suppose that $\val(X)=\{\alpha\}$ for some $\alpha\in \Gamma$. Under these conditions, we first claim that $\rvo$ is not constant on $X$. This is proved by showing that there are $b_i\neq b_j$ such that $\val(b_i-b_j)=\alpha$. If such two elements did not exist, one would have 
\begin{equation*}
\alpha=\val(mb_1)=\val((m-1)b_1-(b_2+\dots +b_m))\geq \min\{\val(b_1-b_k)\ |\ k\neq 1\}>\alpha,
\end{equation*}
a contradiction. Thus, $b_i$ and $b_j$ exist and $\rvo$ is not constant on $X$. For each $\xi\in \rvo(X)$, $1\leq |\rvo^{-1}(\xi)\cap X|<m$. By induction, for some $l\geq 1$ and for each $\xi\in \rvo(X)$ there is a $\xi$-definable injection \mbox{$j_\xi:\rvo^{-1}(\xi)\cap X\lto \rvu^l$}. Then $j:X\lto \rvu^{l+1}$ given by $j(x)=(j_{\rvo(x)}(x),\rvo(x))$ is the desired definable injection.
\end{proof}

The rest of this subsection is to prove the following proposition. It provides  an even more precise description of $\lrv$-formulas. This description is crucial to obtain stratifications of one-dimensional sets.

\begin{Proposition}\label{tstrat-1}
	For any given $L_{\scriptsize \rvu}$-formula $\phi(x, a,\eta)$, where $x$ is a (single)  field variable, and $a$ and $\eta$ are tuples of elements of $R$ and $\rvu$ respectively, there exist a quantifier-free $\lrv$-formula $\phi'(z,\eta')$, with $z$ a tuple of $\rvu$ variables and $\eta'$ a tuple of elements of $\rvu$ (and no extra parameters), and elements $c_0,\dots ,c_m\in R$ such that $\phi(x,a,\eta)$ is equivalent to $\phi'(\rvo(x-c_0),\dots,\rvo(x-c_m),\eta')$.
\end{Proposition}

The strategy to prove this proposition is to start with a partition of the set $X=\phi(R,a,\eta)$ into finitely many intervals and $\val$-discs. Working with the pieces of this partition, we pick the points $c_0,\dots ,c_m$ and the entries of $\eta'$  in such a way that, at the end, whether $x\in R$ is in $X$ is completely determined by equalities and inequalities between the entries of $\eta'$ and the values $\rvo(x-c_0),\dots, \rvo(x-c_m)$. This is clearly enough to conclude the result.

When a piece of the partition is an interval we put the endpoints (or the unique element when a singleton) as points $c_i$'s. When the piece is a $\val$-disc more work is needed: we show that for each cut made by the disc, we can find an $a$-$\lrv$-definable point $d$ \emph{sufficiently close} to the cut; $\rvo(d)$ is then put as an entry of $\eta'$. 

Ensuring the appropriate definability of $d$ above is the major difficulty. We solve this by using~\cite[Lemma]{yin}, which implies that every $a$-$\lrv$-definable closed ball contains an $a$-$\lrv$-definable point. We give a quick overview of the argument.

We fix a (sufficiently) saturated elementary extension $(R^*,\rvu^*)$ of $(R,\rvu)$ (naturally, as  $\lrv$-structures).  
Other substructures of $\sat$ will be denoted as simply $M$, $N$, \dots; we also make use of the notation $\vf(M):=M\cap R^*$ and $\rvu(M):=M\cap \rvu^*$. Notice that being a substructure implies that $\rvo(M):=\{\rvo(x)\ | \ x\in M\}\subseteq \rvu(M)$ and the strict containment could occur.  When the equality holds, i.e. $\rvo(M)=\rvu(M)$, we say that $M$ is \emph{field-generated}.  For $A\subseteq R^*\cup \rvu^*$, $\langle A\rangle$ denotes the $\lrv$-substructure generated by $A$. If $X\subseteq R^*$, $\genl{X}$ denotes the $L$-substructure of $R^*$ generated by $X$. 

\begin{Remark}
If $A\subseteq R^*$, then $\gen{A}$ is field-generated.
\end{Remark}
\begin{proof}
Note that $\vf(\gen{A})$ equals  $\genl{A}$, so it is enough to observe that $(\genl{A},\rvo(\genl{A}))$ is the minimal $\lrv$-substructure containing $A$; this is true because $\rvo(\genl{A})$ will always be contained in any other $\lrv$-structure with field-sort $\genl{A}$. 
\end{proof}

\begin{Fact}
 The substructure $M$ is an elementary substructure of $(R^*,\rvu^*)$ (i.e. a model of $\trv$) if and only if it is field-generated and $\val(M)\neq \{0\}$.
\end{Fact}

Naturally, $R^*$ ---along with its valuation ring $O_{R^*}$--- is a saturated model of $\tcon$. Furthermore, field-generated substructures of $\sat$ correspond naturally to  $\lcon$-substructures of $(R^*,O_{R^*})$. Thus,  arguments about $\lrv$-morphisms between field-generated $\lrv$-substructures can be reduced to facts about $\lcon$-morphisms, and vice versa.

\begin{Definition}[{\cite{yin}}]
For an $\lrv$-substructure $M$ of $(R^*,\rvu^*)$, an $\lrv$-embedding  $\sigma :M\lto (R^*,\rvu^*)$ is said to be immediate if $\sigma(\xi)=\xi$ for all $\xi\in \rvu(M)$. The notion of an immediate $\lrv$-isomorphism between two $\lrv$-substructures $M$ and $N$ is defined in accordance.
\end{Definition}

Notice that if there is an immediate $\lrv$-isomorphism between $M$ and $N$, then, necessarily, $\rvu(M)=\rvu(N)$.

\begin{Remark}
 For any two $a,a'\in R^*$, there exists an immediate $\lrv$-automorphism $\sigma$ of $(R^*,\rvu^*)$ such that $\sigma(a)=a'$ if and only if $a$ and $a'$ have the same type over $\rvu^*$.
\end{Remark}
\begin{proof}
 The implication from left to right is obvious. For the converse, let us assume that $\tp(a/\rvu^*)=\tp(a'/\rvu^*)$ and pick an $\lrv$-automorphism $\sigma$ of $(R^*,\rvu^*)$ such that $\sigma(a)=a'$. If $\sigma$ is not immediate, there is $\xi\in \rvu^*$ such that $\sigma(\xi)\neq \xi$. Let $c\in R^*$ be such that $\rvo(c-a)=\xi$. We must then have $\sigma(\xi)=\rvo(\sigma(c)-\sigma(a))= \rvo(c-a)=\xi$, a contradiction.
\end{proof}

If $M$ is a substructure and $x\in R^*\cup \rvu^*$, $Mx$ denotes the set $M\cup \{x\}$.

\begin{Lemma}\label{onebyone}
	Let $\sigma:M\lto N$ be an immediate $\lrv$-isomorphism. Suppose that $a\in R^*\setminus M$ and $a'\in R^*\setminus N$ are such that $\rvo(a-c)=\rvo(a'-\sigma c)$ for all $c\in M$. Then $\sigma$ can be extended to an immediate $\lrv$-isomorphism $\sigma':\gen{Ma}\lto \gen{Na'}$ such that $\sigma'(a)=a'$. 
\end{Lemma}
\begin{proof}~\cite[Lemma ]{yin}
%
%
\end{proof}

\begin{Corollary}\label{fullauto}
Let $M$ and $N$ be $\lrv$-substructures of $(R^*,\rvu^*)$. Then every immediate $\lrv$-isomorphism $\sigma:M\lto N$ can be extended to an immediate $\lrv$-automorphism $\bar \sigma$ of $(R^*,\rvu^*)$.
\end{Corollary}
\begin{proof} (Proof based on that of{~\cite[Lemma ]{yin}})
Let $\sigma$ be as in the statement; we then know that $\rvu(M)=\rvu(N)$. Let $\xi\in \rvu^*\setminus \rvu(M)$ and $a\in \rvo^{-1}(t)$. Since $a$ is in an open ball entirely disjoint from $M$ and $N$, we deduce that $\rvo(a-c)=\rvo(a-\sigma c)$ for all $c\in \rvu(M)$. By Proposition~\ref{onebyone}, there is an immediate $\lrv$-isomorphism from $\gen{Ma}$ to $\gen{Na}$ extending $\sigma$. By iterating this process over all $\xi\in \rvu^*\setminus \rvu(M)$, we eventually obtain an immediate $\lrv$-embedding $\sigma'$, extending $\sigma$, from an $\lrv$-substructure $M'$ into $\sat$, where $M'$ is such that $\rvu(M')=\rvu^*$. By the quantifier elimination for $\tcon$, $\sigma'|_{R^*}$ can be extended to a full $\lcon$-automorphism of $R^*$. Putting this automorphism on the field-sort and keeping $\sigma'$ on the $\rvu$-sort, we obtain the desired $\lrv$-automorphism $\bar \sigma$.
\end{proof}

If $M$ is a substructure of $\sat$, $\dcl_L(M)$ denotes the $L$-definable closure of $M$, i.e. the set of all elements of $R^*$ definable with an $L$-formula with parameters from $\vf(M)$. 
\begin{Remark}\label{easy5}
Let $M$ be a substructure of $\sat$ and $a,a'\in \vf(M)$. Then, $B(a,\geq \val(a-a'))\cap \dcl_{L}(M)= \varnothing$ if and only if for all $c\in \dcl_L(M)$, $\rvo(c-a)=\rvo(c-a')$.
\end{Remark}
\begin{proof}
For $c\in \dcl_L(M)$, clearly, $\rvo(c-a)\neq \rvo(c-a')$ if and only if $\val(c-a)\geq \val(a-a')$. 
\end{proof}
 
\begin{Proposition}\label{defpointclosedball}
 Let $M$ be a substructure of $\sat$. Any $M$-definable closed ball $B\subseteq R^*$ contains an $M$-definable point.
\end{Proposition}
\begin{proof}
Suppose otherwise, i.e., that $B\cap \dcl_L(M)=\varnothing$. By the saturation of $(R^*,\rvu^*)$, there is an open ball $D\subseteq R^*$ such that $D\cap \dcl_L(M)=\varnothing$ and $B\subsetneq D$. Let $a\in B$ and $a'\in D\setminus B$. Since clearly  $B(a,\geq \val(a-a'))\subseteq D$,  it follows from Remark~\ref{easy5} that $\rvo(c-a)=\rvo(c-a')$, for all $c\in \dcl_L(M)$. By Lemma~\ref{onebyone} and Corollary~\ref{fullauto},  there is an immediate $\lrv$-automorphism $\sigma$ of $(R^*,\rvu^*)$ fixing $M$ and such that $\sigma(a)=a'$. Hence $\sigma(B)\neq B$, which contradicts the $M$-definability of $B$. 
\end{proof}

We finally write down the proof of Proposition~\ref{tstrat-1}.

\begin{proof}[Proof of Proposition~\ref{tstrat-1}]
Let $X:=\phi(R,a,\eta)$. Then $X$ has a finite ($\{a\}\cup \{\eta\}$)-definable partition into intervals and $\val$-discs.  Let $X_0,\dots, X_l$ be the sets of such a partition of $X$.

The collection of points $c_i$ is given by all the endpoints of those $X_i$ that are intervals (or the sole element in $X_i$ when $X_i$ is a singleton). We also add $0$ as one of the points $c_i$. Incidentally, notice that each $c_i$ is in principle ($\{a\}\cup \{\eta \}$)-definable but, however, Proposition~\ref{lemma3} ensures that each $c_i$ is in fact $a$-definable. 

From the $X_i$'s which are $\val$-discs we will obtain the entries of $\eta'$. If $X_i$ is a $\val$-disc, the left cut made by $X_i$ is given by $A_1=\{y\in R\ |\ y< X_i\}$ and $A_2:=\{y\in R\ |\ \exists x\in X_i(x<y)\}$. Let $B$ be the smallest closed ball such that $A_i\cap B\neq \varnothing$ for $i=1,2$. Proposition~\ref{defpointclosedball} implies that there is a ($\{a\}\cup \{\eta\}$)-definable point $b_0$ in $B$ ---which by Proposition~\ref{lemma3} is in fact $a$-definable. We put $\rvo(b_0)$ as one of the entries of the tuple $\eta'$. The right cut of $X_i$ is treated similarly and gives us a point $b_1$ analogous to $b_0$. We add $\rvo(b_1)$ as an entry of the tuple $\eta'$ too. We repeat this process for all such $X_i$'s and, finally, $0_\subrv$ is also added to $\eta'$.

The two paragraphs above describe how to choose the collection $c_0,\dots ,c_m\in R$ and the tuple $\eta'$ of elements of $\rvu$. We now argue that for any $x\in R$, whether $x$ belongs to $X$ depends only on conditions between the values $\{\rvo(x-c_i)\}_{i\leq m}$ and the parameters $\eta'$. It is enough to ilustrate this for each of the possible shapes of the sets $X_i$. If $X_i$ is the singleton $\{c_j\}$, then $x\in X_i$ if and only if $\rvo(x-c_{j})=0_\subrv$; if $X_i=(c_{j},c_{j+1})$, then $x\in X_i$ if and only if $\rvo(x-c_{j})>0_\subrv$ and $\rvo(x-c_{j+1})<0_\subrv$. Now suppose that $X_i$ is a $\val$-disc and let $b_0$ and $b_1$ be the elements previously associated respectively to the left and right cut defined by $X_i$. If both $b_0\in X_i$ and $b_1\in X_i$, then $x\in X_i$ if and only if $\rvo(b_0)\leq \rvo(x)\leq \rvo(b_1)$. If instead $b_0\notin X_i$ and $b_1\in X_i$, then $x\in X_i$ if and only if $\rvo(b_0)< \rvo(x)\leq \rvo(b_1)$. The remaining two combinations are treated similarly. 

We conclude that whether $\phi(x,a,\eta)$ holds is fully determined by a boolean combination of equations and inequalities between the values $\{\rvo(x-c_i)\}_{i\leq m}$ and the parameters $\eta'$. The statement of the proposition follows.
\end{proof}

\section{The Jacobian property and t-stratifications}
\subsection{Basic conditions for t-stratifications and $b$-minimality}
Besides the Jacobian property, a few other basic conditions are employed in~\cite{immi} to ensure the existence of t-stratifications. The following theorem establishes these.  

\begin{Theorem}\label{assumptcon}
 $\then$ satisfies \cite[Hypothesis 2.21 (1)-(3) and (4'')]{immi}. This is, if $(R,\rve)$ is a model of $\then$ and $A\subseteq R\cup \rve$, then, 
\begin{enumerate}
  \item[\textup{(1)}] $\rvu$ is stably embedded in $(R,O_R)$;
  \item[\textup{(2)}] every $A$-definable function $g:\rvu\lto R$ has finite image; 
  \item[\textup{(3)}] for every $A$-definable set $X\subseteq R$ there exists a finite $A$-definable set $S_0\subseteq R$ such that every ball $B\subseteq R$ disjoint from $S_0$ is either contained in $X$ or disjoint from $X$;
 \item[\textup{(4'')}] for any $A$-definable function $f:X\longrightarrow R$ there exists an $A$-definable function $\chi:X\longrightarrow \rve$ such that for each $q\in \chi(X)$, $f|_{\chi^{-1}(q)}$ is either injective or constant.
\end{enumerate}
\end{Theorem}
\begin{proof}
 (1) Suppose that $Q=\{\xi\in \rvu ^n\ |\ \phi(\rvo(t_1(a)),\dots ,\rvo(t_m( a)),\xi,\eta)\}$ where $a$ and $\eta$ are tuples of elements of $R$ and $\rvu$, respectively, and $\phi$ is a quantifier-free $\lrv$-formula in the language of the sort $ \rvu$ (see Convention~\ref{rvterms}). If we set $\eta':=(\rvo(t_1(a)),\dots ,\rvo(t_m(a)))\in \rvu^m$, then clearly $Q$ definable in the sort $\rvu$ by the formula $\phi(\eta',z,\eta)$.\newline
 (2) Set $A_0=A\cap R$ and let $\eta$ be a tuple of elements of $\rvu$ such that $g$ is $(A_0\cup\{ \eta\})$-definable. The result follows from Proposition~\ref{finiteimage}.  \newline
 (3) Le $A_0$ and $\eta$ be as above. Assume that $\phi(x,a,\eta)$ is an $\lrv$-formula defining $X$; by Proposition~\ref{tstrat-1} there are an $\lrv$-formula $\phi'(\xi,\eta)$ and $c_0,\dots ,c_m\in R$ such that the formula 
 \[\phi'(\rvo(x-c_0),\dots ,\rvo(x-c_m),\eta)\]
 defines $X$ too. We claim that taking $\{c_0,\dots,c_m\}$ to be the set $S_0$ is enough. Notice first that, in principle, $S_0$ is only $(A_0\cup \{\eta\})$-definable, but Proposition~\ref{lemma3} ensures that $S_0$ is $A_0$-definable (and thus $A$-definable as needed). Let $B\subseteq R$ be a ball such that $B\cap S_0=\varnothing$ and $B\cap X\neq \varnothing$; we need to show that then $B\subseteq X$. Consider the map $\rho:R\lto \rve$ given by $\rho(x)=\ulcorner (x-c_i)_{i\leq m}\urcorner$. Each fibre of $\rho$ is a singleton or a maximal open ball disjoint from $S_0$. Let $b\in B \cap X$, if $F_b:=\rho^{-1}(\ulcorner (b-c_i)_{i\leq m}\urcorner)$ is a singleton, then $b=c_i$ for some $i\leq m$, which contradicts that $B\cap S_0= \varnothing$. Thus $F_b$ is a maximal open ball disjoint from $S_0$; necessarily, $B\subseteq F_b\subseteq X$ as required. 
 \newline
 (4'') 
  Let $A_0$ and $\eta$ be as before. By Proposition~\ref{JP0lrv}, there is an $(A_0\cup \{\eta\})$-definable map $\chi':R \longrightarrow  \rvu^m$ for which the desired conclusion holds. Suppose that $\varphi(x,y,z)$ is a formula with parameters in $A_0$ such that $\varphi(x,y,\eta)$ defines $\chi'$.  Define the map $\chi'':R \times \rvu^l\longrightarrow \rvu ^m\subseteq \rve$ by declaring $\chi''(x,z)=y$ if and only if $\varphi(x,y,z)$. Notice that then  $\chi''$ is $A_0$-definable and for all $x\in R$, $\chi''(x,\eta)=\chi'(x)$. For $x\in R$, set $p_x:\rvu ^l\longrightarrow \rvu^m$ as $p_x(z)=\chi''(x,z)$ and then consider the canonical parameter $\ulcorner p_x\urcorner$. Since $\rvu$ is stably embedded in $(R,O_R)$, we can assume that $\ulcorner p_x\urcorner$ is an element of $\rve$. Hence the map $\chi:R\longrightarrow \rve$ given by $\chi(x)=\ulcorner p_x\urcorner$ is an $A$-definable map. The result follows since every fibre of $\chi$ is contained in a fibre of $\chi'$.
\end{proof}

The notion of \emph{$b$-minimality with centres} was introduced in~\cite[Section 6]{raf2}.
\begin{Theorem}\label{bmintcon}
 $\then$ is $b$-minimal with centres over (the sorts) $\rve$. This is, if  $(R,\rve)$ is a model of $\then$ and $A\subseteq R\cup \rve$, then
  \begin{enumerate}
  \item[\textup{($b_1'$)}] for every $A$-definable set $X\subseteq R$, there exist $A$-definable functions $\chi:X\longrightarrow \rve$ and $c:\chi(X)\lto R$ such that for each $q\in \chi(X)$, $\chi^{-1}(q)-c(q)$ is an $\rvu$-ball;
  \item[\textup{($b_2$)}] there is no definable surjection from any sort $Q$ in $\rve$ to an open ball in $R$;
  \item[\textup{($b_3$)}] for every $X\subseteq R$ and $A$-definable function $f:X\lto R$ there exists an $A$-definable function $\chi:X\longrightarrow \rve$ such that for each $q\in \chi(X)$, $f|_{\chi^{-1}(q)}$ is either injective or constant. 
  \end{enumerate}
\end{Theorem}
\begin{proof}
 We only need to provide an argument for ($b_1'$). Let $S_0$ be the finite $A$-definable set obtained through Theorem~\ref{assumptcon} (3). We let $\chi:X\lto \rve$ be given by $\chi(x)=\ulcorner \rvo(x-S_0)\urcorner$, for each $x\in X$. Then for any $q\in \chi(X)$, $\chi^{-1}(q)$ is either a singleton $\{s\}$ with $s\in S_0$ or a maximal ball disjoint from $S_0$. By compactness, the centre $c:\chi(X)\lto R$ is defined such that $c(\chi(b))$ satisfies $\val(b-c(\chi(b)))=\max\{\val(b-t)\ |\ t\in S_0 \}$. If $\chi^{-1}(q)$ is a singleton, then clearly $\chi^{-1}(q)-c(q)=\{0\}=\rvo^{-1}(0)$. On the other hand, if $\chi^{-1}(q)$ is an open ball, where, say, $q=\chi(b)$, we have $\chi^{-1}(q)=B(b,>\val(b-c(q)))=\{x\in R\ |\ \rvo(x-c(q))=\rvo(b-c(q))\}=c(q)+\rvo^{-1}(\rvo(b-c(q)))$.
\end{proof}

From Theorem~\ref{bmintcon} we obtain a theorem on cell decomposition \emph{with centres} (\cite[Theorem 6.4]{raf2}).

\begin{Remark}
 Above we proved that Theorem~\ref{assumptcon} (3) implies Theorem~\ref{bmintcon} ($b_1'$). The converse is in fact true. 
\end{Remark}
\begin{proof}
 Suppose that $\chi$ and $c$ are $A$-definable as in ($b_1'$). By the typical manoeuvre we can assume that $\chi$ is $\lrv$-definable; then, by Lemma~\ref{finiteimage} the $A$-definable set $S_0:=c(\chi(X))$ is finite. Let $B\subseteq R$ be any ball disjoint from $S_0$ and suppose that $B\cap X\neq \emptyset$. We want to show that $B\subseteq X$.  First, we claim that for each $x\in B$, $\rvo(x-c(\chi(b)))=\rvo(b-c(\chi(b)))$. If this fails, then there is $x\in B$ such that $\val(b-c(\chi(b)))\geq \val(x-b)$ and so $c(\chi(b))\in B\cap S_0$, a contradiction. To finish, notice that since $b\in \chi^{-1}(\chi((b))$, $\chi^{-1}(\chi(b))-c(\chi(b))=\rvo^{-1}(\rvo(b-c(\chi(b))))$, and so the claim implies that $B\subseteq \chi^{-1}(\chi(b))\subseteq X$.
\end{proof}

\subsection*{The Jacobian property for $n\geq 1$}\label{jacpropsec}
In this section we prove that the full Jacobian property holds for $\then$ ---keeping in mind the power boundedness of $T$. The proof is inspired by that of I. Halupczok of the Jacobian property for valued fields with analytic structure,~\cite[Subsection 5.3]{immi}. The strategy is as follows. We start an inductive argument on $n$. In the inductive step we assume that $n\geq1$ and that $\then$ has the Jacobian property upto $n-1$. By this assumption and Theorem~\ref{assumptcon}, \cite[Theorem 4.12]{immi} ensures that for any $m\leq n$ t-stratifications of definable maps $R^m\lto \rve$ exist. If $f:R^n\lto R$ is an $A$-definable function where $A\subseteq R\cup \rve$, we obtain an $A$-definable map $\rho:R^n\lto \rve$ such that on each fibre $\rho^{-1}(q)$, $f$ equals the restriction of an $L$-definable function, $\jac(f)$ exists and $\rvo(\jac(f))$ is constant. We then refine a t-stratification for $\rho$ so that the fibres of \emph{its rainbow} $\chi$ have a particular description. At the final step, a simple calculation is performed to show that $f$ has the Jacobian property on each pertinent fibre of $\chi$.

The following is an easy lemma, analogue of~\cite[Lemma 5.8]{immi}, which allows us to eventually make the Jacobian property rest on properties of $L$-definable functions.
\begin{Lemma}\label{lemma5.8}
Let $g:O_R\longrightarrow O_R$ be an $L$-definable differentiable function such that for all $x\in O_R$, $g'(x)\in O_R$ and $\res(g'(x))$ is constant. Then for any $x,x'\in O_R$ with $x\neq x'$, 
\begin{equation}
\val(g(x)-g(x')-g'(0)(x-x'))>\val(x-x'). \notag 
\end{equation}
\end{Lemma}
\begin{proof}
 We may assume that $g'(0)=0$, otherwise we can replace $g$ by the function $g(x)-g'(0)x$. Then for any $x\in O_R$, $\val(g'(x))>0$. By the Mean Value Theorem, for any $x$ and $x'$ as in the hypothesis, there exists $x''\in O_R$ such that $g(x)-g(x')=g'(x'')(x-x')$. Hence $\val(g(x)-g(x'))>\val(x-x')$ as required.
\end{proof}

 Let $B_0\subseteq K^n$(typically an open ball or the whole of $R^n$). A t-stratification of a map $\rho:B_0\lto \rve$ is an  $\lhen$-definable partition $S_0,\dots, S_n$ of $B_0$ such that for each $d\leq n$: (1) $\dim(S_d)\leq d$ and (2) on each ball $B$ disjoint from $S_0\cup \dots \cup S_d$, the characteristic maps\footnote{The characteristic map of $X$ is the function sending elements of $X$ to 1 and non-elements to 0.} $1_{S_0},\dots ,1_{S_n}$ along with $\rho$ are \emph{almost translation invariant} in the direction of a $d$-dimensional subspace of $R^n$. In condition (2) we say that the stratification $S_0,\dots S_n$ \emph{reflects} $\rho$. A t-stratifications of a set is a t-stratification of its characteristic map. The reader is referred to~\cite{immi} for full details and to~\cite{immi2} for an accessible introduction to t-stratifications; since our focus here is on proving the Jacobian property, we limit ourselves to cite results from~\cite{immi} where needed. We denote t-stratifications as $(S_i)_{i\leq n}$ and, for each $d\leq n$, we use $S_{\leq d}$ and $S_{\geq d}$ to denote $S_0\cup \dots \cup S_d$ and $S_d\cup \dots \cup S_n$ respectively.

We fix a set of parameters $A\subseteq R\cup \rve$ and an $A$-definable set $B_0\subseteq R^n$. The \emph{rainbow} of the t-stratification $(S_i)_{i\leq n}$ of $B_0$ is the map $B_0\lto \rve$ given by $x\longmapsto \ulcorner (x-S_i)_{i\leq n} \urcorner$. Observe that if $(S_i)_{i\leq n}$ is $A$-definable, so is its rainbow. If $V$ is a vector subspace of $\overline R^n$, a coordinate projection $\pi:R^n\lto R^d$ is said to be an \emph{exhibition} of $V$ if $\res\circ \pi$ is an isomorphism between $V\cap O^d_R$ and $\overline R^d$. The \emph{affine direction} of $X\subseteq R^n$ is the subspace of $\overline R$, denoted as $\afd(X)$, generated by the set $\{\res(x-y)\ |\ x,y\in X\cap O_R^n\}$.  By~\cite[Lemma 4.3]{immi}, if $C\subseteq S_j$ is a fibre of the  rainbow of the $A$-definable t-stratification $(S_i)_{i\leq n}$ and $\pi$ is an exhibition of $\afd(C)$, then there is an $A$-definable map $c:\pi(C)\longrightarrow R^{n-j}$ such that $C=\graph(c)$. We keep this notation in the following definition. 

\begin{Definition}
  Let $(S_i)_{i\leq n}$ be a t-stratification (of $B_0$) and $d\leq n$. We say that $(S_i)_{i\leq n}$ has the property $(*)_d$ if for any $j\geq d$ and fibre $C\subseteq S_j$ of the rainbow of $(S_i)_{i\leq n}$, the corresponding function $c:\pi(C)\longrightarrow R^{n-j}$ is the restriction of an $L$-definable differentiable function to $\pi(C)$.
\end{Definition}

When not specified otherwise, by $(S_i)_{i\leq n}$ being a t-stratification we mean that $(S_i)_{i\leq n}$ is a t-stratification of the fixed set $B_0$.
\begin{Lemma}
 Assume that $\then$ has the Jacobian property upto $n-1$. If the $A$-definable t-stratification   $(S_i)_{i\leq n}$ has property $(*)_{d+1}$, then there is an $A$-definable t-stratification $(S'_i)_{i\leq n}$ that respects $(S_i)_{i\leq n}$ and has property $(*)_{d}$.
\end{Lemma}
\begin{proof}
Let $C\subseteq S_d$ be a fibre of the rainbow of $(S_i)_{i\leq n}$, $\pi$ and $c:\pi(C)\longrightarrow R^{n-d}$ as above. By Proposition~\ref{sharperwomono},  there is a $\ulcorner C\urcorner$-definable map $\chi_{\pi}:\pi(C)\longrightarrow \rve$ such that on any $\chi_\pi$-fibre $c$ is the restriction of an $L$-definable differentiable function. By composing with $\pi:C\longrightarrow \pi(C)$ we obtain a definable map $C\longrightarrow \rve$. Now we do this for all exhibitions of $\afd(C)$, obtaining a collection of maps. We let $\chi _{_C}:C\longrightarrow \rve$ be the product of all these maps. Note that $\chi_{_C}$ is $\ulcorner C\urcorner$-definable. We perform this construction for all the fibres of the rainbow of $(S_i)_{i\leq n}$ contained in $S_d$. Using that these fibres cover the whole of $S_d$ and a compactness argument, we obtain a single $A$-definable map $\chi:S_d\longrightarrow \rve$ such that for any $C, \pi$ and $c$ as above and fibre $X\subseteq C$ of $\chi$, the map $c|_{\pi(X)}$ equals the restriction of an $L$-definable differentiable function to $\pi(X)$. 

By Theorem~\ref{assumptcon}, the assumption on $\then$ in the hypothesis and~\cite[Theorem 4.12]{immi}, there is an $A$-definable t-stratification $(T_i)_{i\leq n}$ reflecting both $(S_i)_{i\leq n}$ and $\chi$.  By~\cite[Theorem 4.20]{immi}, setting $S'_{\leq j}:=T_{\leq j}$ for $j<d$ and $S'_{\leq j}:=S_{\leq j}\cup T_{\leq d-1}$ for $j\geq d$, defines an $A$-definable t-stratification reflecting both $(S_i)_{i\leq n}$ and $\chi$ and satisfying that $S_j'\subseteq S_j$ whenever $j\geq d$. We claim that $(S_i')_{i\leq n}$ has property $(*)_d$. 

Let $C'\subseteq S_j'$ be a fibre of the rainbow of $(S_i')_{i\leq n}$, for some $j\geq d$. Then there is a fibre $C\subseteq S_j$ of the rainbow of $(S_i)_{i\leq n}$, such that $C'\subseteq C$. It then follows that $\afd(C')=\afd(C)$ and, if $\pi$ is an exhibition of this subspace, the corresponding map $c':\pi(C')\longrightarrow R^{n-j}$ is the restriction of the map $c:\pi(C)\longrightarrow R^{n-j}$ to $\pi(C')$. If $j>d$, by property $(*)_{d+1}$ for $(S_i)_{i\leq n}$, $c$ is the restriction to $\pi(C)$ of an $L$-definable differentiable function, which then obviously implies a similar conclusion for $c'$. If instead $j=d$, since $(S_i')_{i\leq n}$ reflects $\chi$, $\chi$ is constant on $C'$ and the construction of $\chi$ ensures that $c'$ is the restriction to $\pi(C')$ of an $L$-definable differentiable function.  
\end{proof}

Notice that by the proof above $(S_i')_{i\leq n}$ can be taken to be a refinement of $(S_i)_{i\leq n}$, i.e. for each $j\leq n$, $S_j'\subseteq S_j$.

\begin{Remark}\label{fullrefinement}
Under the same assumptions in the last lemma, for any $A$-definable t-stratification $(S_i)_{i\leq n}$ there exists an $A$-definabe t-stratification $(S_i')_{i\leq n}$ reflecting and refining $(S_i)_{i\leq n}$ for which, for any $j\leq n$, fibre $C\subseteq S_j'$ of the rainbow of $(S_i')_{i\leq n}$ and exhibition $\pi$ of $\afd(C)$, the map $c:\pi(C)\longrightarrow R^{n-j}$ is the restriction of an $L$-definable differentiable map to $\pi(C)$.
\end{Remark}

The `almost'  in \emph{almost translation invariant} stands for the allowance of some rigid maps, which, after being applied, actualise translation invariability. These maps are bijections $\varphi:X\lto Y$, where $X,Y\subseteq R^n$, such that $\rvo(\varphi(x)-\varphi(y))=\rvo(x-y)$ for all $x,y\in X$. Each such map is called a \emph{risometry} between $X$ and $Y$. A risometry $\varphi:X\lto X$ \emph{respects} a map $\rho:X\lto \rve$ if $\rho\circ \varphi=\rho$.  \gl{n} stands for the set of $n\times n$ invertible matrices with entries from $O_R$. 

\begin{Proposition}\label{refinerainbow}
 Assume that $\then$ has the Jacobian property upto $n-1$. Let $(S_i)_{i\leq n}$ be an $A$-definable t-stratification. Then there is an $A$-definable t-stratification $(S_i')_{i\leq n}$ refining $(S_i)_{i\leq n}$ such that whenever $C\subseteq S_n'$ is a fibre of the rainbow of $(S_i')_{i\leq n}$, there exist open balls $B_1,\dots, B_n\subseteq R$ and an $L$-definable differentiable bijection $h:B_1\times \dots \times B_n\longrightarrow C$ which can be written as a composition of a risometry and a matrix in \gl{n}. 
\end{Proposition}
\begin{proof}
 Let $(S_i')_{i\leq n}$ be as described in Remark~\ref{fullrefinement}. Let $C\subseteq S_n'$ be a fibre of its rainbow. Suppose that $S_0'\neq \emptyset$ (we could set $S_0'=\{x_0\}$ for any $A$-definable point $x_0\in B_0$ if necessary). Then by~\cite[Lemma 4.21]{immi}, $C$ is entirely contained in a ball $B\subsetneq B_0$ and by taking $B$ maximal one ensures that $\tsp_B((S_i')_{i\leq n})=\{0\}$. This situation corresponds to the case $d=0$ (with $\pi_0:R^n\longrightarrow R^0\ (=\{0\}$) and $\lambda:=\rad(B)$) in the following set of conditions. For $d\leq n$, there are a coordinate projection $\pi=\pi_d:R^n\longrightarrow R^d$ and $\lambda=\lambda_d\in \Gamma$ such that:
 \begin{enumerate}[(i)]
  \item for every $q\in \pi(C)$, there is a ball $B_q$ (regarded as a subset of $\{0\}^{d}\times R^{n-d}$) of radius $\lambda$ such that $C\cap \pi^{-1}(q)\subseteq B_q\subseteq \pi^{-1}(q)$;
  \item for any $q,q'\in \pi(C)$, there is a definable risometry between $B_q$ and $B_{q'}$ respecting the rainbow of $(S_i')_{i\leq n}$;
  \item if $B\subseteq R^n$ is the ball of radius $\lambda$ containing $B_q$, then $\pi$ exhibits $\tsp_B((S_i')_{i\leq n})$;
  \item there are open balls $B_1,\dots, B_d\subseteq R$ and an $L$-definable differentiable bijection $h:B_1\times \dots \times B_d\longrightarrow \pi(C)$ which can be written as the composition of a risometry and a matrix in \gl{d}.
 \end{enumerate}
 Clearly, when $d=n$ (iv) gives the result.  Thus we work inductively, recalling that as stated before these conditions hold when $d=0$.  Suppose that for a fixed $d<n$, there are $\pi$ and $\lambda$ as above satisfying conditions (i)-(iv). We show that there is $d'>d$ for which there are $\pi':R^n\lto R^{d'}$ and $\lambda'\in \Gamma$ such that (i)-(iv) hold.

 For notational simplicity, we assume that $\pi$ is the projection to the first $d$ coordinates. Since $\pi$ exhibits $\tsp_B((S_i')_{i\leq n})$ and by~\cite[Lemmas 3.16 and 4.21]{immi}, for a fixed $q\in \pi(C)$ the collection $(S_i'\cap B_q)_{d\leq i\leq n}$ is a t-stratitication of $B_q$ and $C\cap B_q$ is a fibre of its rainbow. Notice that $(C\cap B_q)\cap S_d'=\emptyset$ (because $C\cap S_d'=\emptyset$), hence~\cite[Lemma 4.21]{immi} implies that $C\cap B_q$ is contained in a ball $D_q\subseteq B_q\setminus S_d'$. We further assume that such $D_q$ is maximal. Since $S_d'\cap B_q$ is finite and nonempty, $D_q$ is a fibre of the map $x\mapsto \rvo(x-S_d'\cap B_q)$. If we let $s_q$ be any element of $S_d'\cap B_q$, then $ \rvo(x-s_q)$ is constant on $D_q$; let $\xi_q\in \text{RV}^{n-d}$ be this constant value. Hence $D_q=s_q+(\{0\}^d\times  \rvo^{-1}(\xi_q))$. For any other $q'\in \pi(C)$, applying the risometry in (ii) (when applied to $d,\pi$ and $\lambda$) from $B_q$ to $B_{q'}$ provides us with $D_{q'}$ and $s_q'\in S_d'\cap B_{q'}$ with analogous properties. For the balls $D_{q'}$ the corresponding $\xi_{q'}$ equals $\xi_q$ by construction, so all of these balls have radius $\lambda ':=\vrv(\xi_q)$. Additionally, since the risometries in (ii) respect the rainbow of $(S_i')_{i\leq n}$, the collection $\{s_q\}_{q\in \pi(C)}$ is contained in a fibre $\tilde C\subseteq S_d'$ of the rainbow of $(S_i')_{i\leq n}$, and if $q,q'\in \pi(C)$, $\tsp_{B_q}((S_i'\cap B_q)_{d\leq i\leq n})=\tsp_{B_{q'}}((S_i'\cap B_{q'})_{d\leq i\leq n})$; we denote this common space by $V$ and let $\tau:R^{n-d}\longrightarrow R^{\dim(V)}$ be an exhibition of it. Since $S_d'\cap D_q=\emptyset$ (for any $q\in \pi(C)$), there is at least 1-translatability of $(S_i'\cap B_q)_{d\leq i\leq n}$ on $D_q$ (see~\cite[Section 3.1]{immi}), so $\dim(V)\geq 1$. We set $d':=d+\dim(V)$ and $\pi':R^n\longrightarrow R^{d'}$ as $\pi'(x)=\pi(x)\oplus \tau(x-\pi(x))$.

 We now show that $d', \lambda'$ and $\pi'$ as defined above satisfy the corresponding versions of conditions (i)-(iv). For notational simplicity we assume that $\pi'$ is the projection to the first $d'$ coordinates on $R^n$ (this is of course compatible with our previous assumption on $\pi$). For (i), if $q=(q_1,\dots ,q_{d},\dots, q_{d'})\in \pi'(C)$, let $B_q':=B(x,> \lambda')\subseteq \pi'^{-1}(q)$, where $x\in \pi'^{-1}(q)$ is any element such that  $\pi(x)\in C\cap B_{(q_1,\dots, q_d)}$. Then clearly, $C\cap \pi'^{-1}(q)\subseteq B_q'\subseteq \pi'^{-1}(q)$. This defines balls $B_q'$ ($q\in \pi'(C)$) for condition (i). For each $q\in \pi'(C)$, $B_q'$ is contained in $B_{\pi(q)}$, and therefore the risometry from $B_{\pi(q)}$ to $B_{\pi(q')}$ restricts to a risometry from $B_q'$ to $B_{q'}'$ for all $q,q'\in \pi(C)$; the risometry so defined fulfills condition (ii). Since $\pi$ is an exhibition of $\tsp_{B_{\pi(q)}}((S_i')_{i\leq n})\subseteq k^{d}$ and $\rho$ of $V\subseteq k^{n-d}$, the construction of $\pi'$ ensures that it is an exhibition of $\tsp_{B'}((S_i')_{i\leq n})\subseteq k^{d'}$, where $B'\subseteq R^n$ is the $n$-dimensional ball of radius $\lambda'$ containing $B_q'$. So far (i)-(ii) have been established; we prove (iv) in the next paragraph.
 
 Since $\afd(\tilde C)$ is exhibited by $\pi$, the properties of $(S_i')_{i\leq n}$ imply that the function $\tilde c:\pi(\tilde C)\longrightarrow R^{n-d}$ corresponding to $\tilde C$ is the restriction of an $L$-definable differentiable function. Notice that by the choice of $\tilde C$ it is clear that $\pi(\tilde C)=\pi(C)$; we use this fact without further comment below. Since $D_q=s_q+(\{0\}^d\times  \rvo^{-1}(\xi_q))$ and $s_q=(q,\tilde c(q))\in \tilde C$, $\pi'(D_q)=\{q\}\times (\tau(\tilde c(q))+\tau( \rvo^{-1}(\xi_q)))$. Let $U$ denote the product of open balls $B_1\times \dots\times B_d\times \tau( \rvo^{-1}(\xi_q))\subseteq R^{d'}$. We define $h':U\longrightarrow \pi'(C)$ by $h'(x,y)=(h(x), \tau\circ \tilde c\circ h(x)+y)$, for any $x\in B_1\times \dots \times B_d$ and $y\in \tau( \rvo ^{-1}(\xi_q))$. Clearly $h'$ is $L$-definable and differentiable, and, using that $h$ is so, it is easy to see that $h'$ is bijective. To finish the proof we need to check that $h'$ can be written as the composition of a risometry and a matrix in \gl{d'}. By~\cite[Lemma 4.3]{immi}, the map $(q,y)\mapsto (q,\tilde c(q)+y)$ from $\pi(C)\times R^{n-d}$ to itself equals the composition of a risometry and a matrix in \gl{n}. After composing with $\pi'$, we get that the map $(q,y)\mapsto (q,\tau(\tilde c(q))+y)$ from $\pi(C)\times \tau( \rvo^{-1}(\xi_q))$ to itself is the composition of a risometry and a matrix in \gl{d'}. Finally, adding that $h:B_1\times \dots\times B_d\longrightarrow \pi(C)$ is already a composition of a risometry and a matrix in \gl{d}, the required property of $h'$ follows.
\end{proof}
We finally prove the main result of the section---and of the paper. 

\begin{Theorem}\label{jpforthen}
  Any definable function in $(R, \rve)$ has the Jacobian property. This is, if $A\subseteq R\cup \rve$, $B_0\subseteq R^n$ and  $f:B_0\lto R$ is  $A$-definable in $(R,\rve)$, then there exists an $A$-definable map $\chi:B_0\lto S\subseteq \rve$  such that if $q\in S$ satisfies that $\dim(\chi^{-1}(q))=n$, then there exists $z\in R^n\setminus \{0\}$ such that 
\begin{equation}\label{eq:jpeq}
\val(f(x)-f(x')-\langle z,x-x'\rangle)>\val(z)+\val(x-x'),
\end{equation}
for all distinct $x,x'\in \chi^{-1}(q)$.
\end{Theorem}
\begin{proof} This proof is based on the second part of the proof of~\cite[Proposition 5.12]{immi}. Properties of analytic functions used in that proof are substituted here by properties of $\lrv$- and $L$-definable functions. 

 We perform induction on $n$. The base case $n=0$ is trivial. Suppose that $n\geq 1$ and that, for any $m\leq n-1$, any definable function from $R^{m}$ into $R$ has the Jacobian property. Let $A,B_0$ and $f$ be as in the hypotheses. By the typical argument, we assume that $f$ is $\lrv$-definable. Let $C_1,\dots,C_l$ be the $\lrv$-definable partition of $B_0$ and $f_1,\dots ,f_l$ the respective $L$-definable functions given by Proposition~\ref{sharperwomono}.  Let $\jac(f_j)$ be the Jacobian of $f_j$ (which is at least defined on $C_j$ because $f|_{C_j}=f_j|_{C_j}$).  In general, the local dimension of a set $X$ at a point $x\in X$ is defined as $\dim_x(X):=\min \{\dim(X\cap D)\ |\ x\in D\text{ is an open ball}\}$. In our context the following  general property of dimension holds for any set $X\subseteq R^m$: $\dim(\{x\in X\ | \ \dim_x(X)<\dim(X)\})<\dim(X)$ (see~\cite[Theorem 3.1]{forn}). Hence, for each $j\leq l$, $\{x\in C_j\ |\ \dim_x(C_j)<\dim(C_j)\}$ has dimension strictly less than $\dim(C_j)$, so we can assume that $C_j$ is of constant local dimension $\dim(C_j)$ (i.e. $\dim_x(C_j)=\dim(C_j)$ for each $x\in C_j$). From now on we assume that $\dim(C_j)=n$ (the only case we are interested in relation to the Jacobian property). After further refining the partition, we obtain an $A$-definable map $\rho:X\longrightarrow \rve$ for which each fibre is contained in a set $C_j$ and on any of its $n$-dimensional fibres, $ \rvo(\jac(f))$  ($=\rvo(\jac(f_j))$) is constant.

 By the inductive hypothesis and~\cite[Theorem 4.12]{immi}, there exists an $A$-definable t-stratification reflecting $\rho$. From such  stratification we obtain a t-stratification $(S_i')_{i\leq n}$ as in the statement of Proposition~\ref{refinerainbow}. Notice that the rainbow $\chi$ of $(S_i')_{i\leq n}$ refines the map $\rho$, so $ \rvo(\jac(f))$ is defined and constant on any $n$-dimensional fibre of $\chi$. Let $C$ be an  $n$-dimensional fibre of $\chi$ and $B_1,\dots, B_n$ and $h:B_1\times \dots \times B_n\longrightarrow C$ be as in  Proposition~\ref{refinerainbow}. Fix $\varphi$ a risometry and $M$ a matrix in \gln\ such that $h=\varphi \circ M$.

 Let $\xi$ be the (unique) value of $\rvo(\jac(f))$ on $C$.  If $\xi=0$,  $\jac(f|_C)(x)=0$ for all $x\in C$ and hence $f$ is constant on $C$, leaving nothing to prove. Consequently, from now on we assume that $\xi\neq 0$ and pick $z\in \rvo^{-1}(\xi)$.  Let $x,x'\in C$ be distinct. We use Lemma~\ref{lemma5.8} and properties of $h$ to show~\eqref{eq:jpeq}.

 We define the function $\eta:O_R\longrightarrow B_1\times \dots \times B_n$ by $\eta(t):=th^{-1}(x)+(1-t)h^{-1}(x')$. Then the composition $\theta:=h\circ \eta:O_R\longrightarrow C$ is a definable differentiable function. Notice that for $t\neq t'$ in $O_R$, 
 $
 M\circ \eta(t)-M\circ\eta(t')=(t-t')(h^{-1}(x)-h^{-1}(x')),\notag
 $
 and hence
 \begin{equation}
 \rvo\left(\frac{\theta(t)-\theta(t')}{t-t'}\right)= \rvo\left(\frac{M\circ \eta(t)- M\circ\eta(t')}{t-t'}\right)= \rvo(h^{-1}(x)-h^{-1}(x')), \notag
 \end{equation}
 meaning that $ \rvo\left((\theta(t)-\theta(t'))/(t-t')\right)$ is constant on $O_R^2\setminus \{(t,t)\ | \ t\in O_R\}$. This implies that $ \rvo(\jac(\theta))$ is constant on $O_R$, and it equals $\rvo\left((\theta(t)-\theta(t'))/(t-t')\right)$ for any distinct $t,t'\in O_R$. In particular, by taking $t=1$ and $t'=0$, we deduce that
 \begin{equation}\label{eq:1}
 \val(x-x'-\jac(\theta)(0))>\val(x-x').
 \end{equation}
 We set $g:=f\circ \theta:O_R\longrightarrow R$. Then $g$ is a definable diferentiable function and, by the chain rule, $g'(t)=\langle \jac(f|_C)(h\circ \eta(t)),\jac (\theta)(t)\rangle$ for all $t\in O_R$. It follows that for any $t\in O_R$, $\val(g'(t))\geq \val(\jac(f|_C))+\val(\jac(\theta)).$
 Similarly, for $t,t\in O_R$ with $t\neq t'$,
 \begin{align*}
 \val(g'(t)-g'(t')) \geq & \val(\jac(f|_C)(h\circ \eta(t))-\jac(f|_C)(h\circ \eta(t')))\\ &+\val(\jac(\theta)(t)-\jac(\theta)(t')) \\
  > &\val(\jac(f|_C))+\val(\jac(\theta)).
 \end{align*}
 Let $r\in R$ be such that $\val(r)=\val(\jac(f|_C))+\val(\jac(\theta))$ and notice that then the function $t\longmapsto g(t)/r$ satisfies all the conditions of Lemma~\ref{lemma5.8}. Thus, for all distinct $t,t'\in O_R$,
 $
 \val(g(t)-g(t')-g'(0)(t-t'))> \val(t-t')+\val(\jac(f|_C))+\val(\jac(\theta)).\notag
 $
 By taking again $t=1$ and $t'=0$ we obtain,
 \begin{equation}\label{eq:2}
  \val(f(x)-f(x')-g'(0))> \val(\jac(f|_C))+\val(\jac(\theta)).
 \end{equation}
  By the choice of $z$, $\val(z)=\val(\jac(f|_C))$ and by~\eqref{eq:1}, $\val(\jac(\theta))=\val(x-x')$. On the other hand, $g'(0)=\langle \jac(f|_C)(x'),\jac(\theta)(0)\rangle =\langle z,\jac(\theta)(0)\rangle$. Hence,~\eqref{eq:2} becomes
 \begin{align*}
 \val(f(x)-f(x')-\langle z,\jac(\theta)(0)\rangle)> \val(z)+\val(x-x').  
 \end{align*}
 Lastly, from~\eqref{eq:1} we have
 \begin{align*}
  \val(\langle z,\jac(\theta)(0)\rangle-\langle z,x-x'\rangle)> \val(z)+\val(x-x').
 \end{align*}
 Combining the last two inequalities proves~\eqref{eq:jpeq}.
\end{proof}

We are now able to conclude that t-stratifications exists for any definable map $B_0\lto \rve$. We write an abstract version of this for later reference. 

\begin{Corollary}\label{existence2}
 Let $\phi$ be an $\lhen$-formula defining a map $\chi_\phi(R):R^n\lto \rve$ in all models $(R,\rve)$ of $\then$. Then there are $\lhen$-formulas $\psi_0,\dots,\psi_n$ such that in each model $(R,\rve)$ of $\then$, the family of sets $(\psi_i(R))_{i\leq n}$ is a t-stratification reflecting $\chi_\phi(R)$. 
\end{Corollary}
\begin{proof}
 It has been proved that $\then$ satisfies~\cite[Hypothesis 2.21 (1)-(3) and (4'')]{immi} and that definable functions on its models have the Jacobian property. The result follows from~\cite[Theorem 4.12 and Corollary 4.13]{immi}.
\end{proof}

\subsection{No t-stratifications under arbitrary $T$}\label{notstrats}
The following example shows that the hypothesis of power boundedness of $T$ is essential for the existence of t-stratifications.

Let $T$ be a non-power bounded o-minimal theory expanding the $\lori$-theory of real closed fields, and $R$ one of its models. Let $g:R_{>0}^2\lto R$ be the function given by $g(x,y)=x^y$. From the dichotomy in o-minimal theories (see~\cite{miller}) an exponential map\footnote{i.e. a group isomorphism between the additive group $(R,+,0)$ and the multiplicative group $(R_{>0}, \cdot, 1)$.} is definable in $R$. It follows that $g$ is always $\lhen$-definable in any model $(R,\rve)$ of $\then$. 

Due to the growth-rate of exponential maps, for fixed $a,b>0$, no risometry can simultaneously \emph{straighten} both the sets $\{(a,y,g(a,y))\ |\ y\in R_{>0}\}$ and $\{(b,y,g(b,y))\ |\ y\in R_{>0}\}$. The following can be easily concluded. 
\begin{Proposition}
 The graph of $g$ above does not admit a t-stratification. 
\end{Proposition}

Thus, the main result in the last section---that t-stratifications exist for definable sets in models of $\then$ whenever $T$ is power bounded---is the sharpest we could have expected. 
%

\section{Applications}
\subsection{Application 1: Stratifications induced on tangent cones}
If $X\subseteq R^n$ and $p\in R^n$, the \emph{tangent cone of $X$ at $p$} is the set 
\[
\mathcal C_p(X):=\{y\in R^n\ | \ \forall \gamma\in \Gamma \exists x\in X,r\in R_{>0}(\val(x-p)>\gamma \wedge \val(r(x-p)-y)>\gamma)\}.
\]
The tangent cone equals the tangent space of $X$ at $p$ whenever the latter exists (i.e. when $p$ is a regular point of $X$). 

Fix a set $X$ and a point $p$ as above. A t-stratification $(S_i)_{i\leq n}$ of $X$ induces a stratification of $\mathcal C_p(X)$ given by the sets $\mathcal C_{p,0}:=\mathcal C_p(S_0)$ and $\mathcal C_{p,i+1}:=\mathcal C_{p}(S_0\cup\dots\cup S_{i+1})\setminus \mathcal C_{p}(S_0\cup\dots\cup S_{i})$ for $0\leq i<n$. In~\cite[Section 4]{egr} the following result was foreseen.

 \begin{Theorem}\label{intro1}
  If $X$ is an $\lhen$-definable subset of $R^n$, $p$ is a point in $R^n$, and $(S_i)_{i\leq n}$ is a t-stratification of $X$, then the collection $(\mathcal C_{p,i})_{i\leq n}$ constitutes a t-stratification of  $\mathcal C_p(X)$.
 \end{Theorem}
 \begin{proof}
 	See the argument for~\cite[Theorem 4.1]{egr}. There, the result is conditional on the existence of t-stratifications for definable sets in models of $\then$ (reminder: we have not dropped the assumption of power boundedness of $T$), but this has been proved in this work. Thus, the theorem holds without further assumptions.  
 \end{proof}
 
\subsection{$L$-definable t-stratifications}

Corollary~\ref{existence2} states the existence of t-stratifications definable in the language $\lhen$. Here  we will prove that for any given t-stratification $(S_i)_{i\leq n}$ (say, of $B_0$) we can always find  another t-stratification $(S_i')_{i\leq n}$ reflecting any map reflected by $(S_i)_{i\leq n}$ and for which each set $S_{i}'$ is $L$-definable. When $L=\lori$ this result is in~\cite[Section 6]{immi}. 

The next proposition---a version of~\cite[Proposition 6.2]{immi}---uses the following notion of dimension of formulas. 

\begin{Definition}
 Let $x$ be a tuple of field sort variables and $\phi(x)$ an $\lhen$ formula. The \emph{dimension of $\phi(x)$} is defined as 
  \begin{equation}
    \dim(\phi(x)):=\max\{\dim(\phi(R))\ |\ (R,\rve) \text{ is a model of }\then\}.\notag
  \end{equation}
\end{Definition}

\begin{Proposition}\label{deltaset}
Let $\Delta$ be a family of $\lhen$-formulas with the following properties:
  \begin{enumerate}
    \item[\textup{(i)}] $\Delta$ is closed under disjunctions and contains $\bot$;
    \item[\textup{(ii)}] for each $\lhen$-formula $\phi$ there is a formula $\phi^\Delta\in \Delta$ such that $\phi \rightarrow \phi^\Delta$ and $\dim(\phi) =\dim (\phi^\Delta)$. 
  \end{enumerate}
Also assume that $(\phi _i)_{i\leq n}$ is a tuple of $\lhen$-formulas defining a t-stratification in all models of $\then$. Then there exists a tuple of $\lhen$-formulas $(\psi_i)_{i\leq n}$ that defines a t-stratification reflecting $(\phi_i)_{i\leq n}$ in all models of $\then$ and such that, for each $i\leq n$, $\psi_{\leq i}:=\psi_0\vee \dots \vee \psi_i$ is equivalent to a formula in $\Delta$.  
\end{Proposition}
\begin{proof} 

Let $(\phi_i)_{i\leq n}$ be a tuple of formulas as in the hypotheses and pick formulas $\phi_i^\Delta$ as in (ii). Set, for each $i\leq n$, $\phi_{\leq i}:=\phi_0\vee \dots \vee \phi_i$ and similarly consider $(\phi_{\leq i})^\Delta\in \Delta$. Fix $d\leq n$ and suppose that for all $i\leq n$, $\dim((\phi_{\leq i})^\Delta\wedge \neg (\phi_{\leq i}))\leq d$. A repeated application of the following claim establishes the proposition.\newline
{\bf Claim.} There exist a t-stratification $(\phi'_i)_{i\leq n}$ reflecting $(\phi_i)_{i\leq n}$ and formulas $\phi_{0}^*,\dots, \phi_{n}^*$ such that, for each $i\leq n$, $\phi _{\leq i}^*\in \Delta, \phi'_{\leq i}\rightarrow \phi _{\leq i}^*$ and $\dim(\phi^*_{\leq i}\wedge \neg (\phi '_{\leq i}))\leq d-1$.\newline
\textit{Proof of Claim.} Set $\delta_n:=(\phi_{\leq n})^\Delta \wedge \neg (\phi_{\leq n})$ and pick $\delta_n^{\Delta}\in \Delta$ as in (ii). Fix $i<n$ and suppose that $\delta_{i+1}$ is defined already and that we have made a choice of $\delta_{i+1}^\Delta$ too. We then set \[\delta_i:=((\phi_{\leq i})^\Delta\wedge \neg (\phi_{\leq i}))\vee (\delta _{i+1}^\Delta\wedge \delta _{i+1})\vee \dots \vee (\delta _{n}^\Delta\wedge \neg \delta _{n}).\] 
This process defines the formulas $\delta _i$ recursively, along with the formulas $\delta_i^\Delta$, $i\leq n$. Notice that $\dim(\delta_i)\leq d$ for all $i\leq n$. If $\delta :=\bigvee _{i\leq n}\delta_i$, then for any $i\leq n$, the formula $\phi _{\leq i}\vee \delta $ is equivalent to a disjunction of formulas in $\Delta$ as indicated below. We use $\equiv$ to denote logical equivalence.
\begin{flalign*}
  \phi _{\leq i}\vee \delta \equiv&  \bigvee _{j\leq i}\left(\phi_{\leq j}\vee \delta _j\vee\dots \vee \delta _n\right) & \\
  =&\bigvee _{j\leq i}[\phi_{\leq j}\vee ((\phi_{\leq j})^\Delta\wedge \neg (\phi_{\leq j}))\vee (\delta _{j+1}^\Delta\wedge \neg\delta _{j+1})\vee \dots \vee (\delta _{n}^\Delta\wedge \neg \delta _{n})& \\ &\vee \delta _{j+1}\vee\dots \vee \delta _n ]& \\
  =&\bigvee _{j\leq i}[(\phi_{\leq j}\vee ((\phi_{\leq j})^\Delta\wedge \neg (\phi_{\leq j})))\vee ((\delta _{j+1}^\Delta\wedge \neg\delta _{j+1})\vee \delta _{j+1})\vee \dots& \\& \vee ((\delta _{n}^\Delta\wedge \neg \delta _{n})\vee \delta _n)] & \\
  \equiv& \bigvee _{j\leq i}((\phi_{\leq j})^\Delta\vee \delta_{j+1}^\Delta\vee \dots \vee \delta_n^\Delta).
\end{flalign*}
Now let $(\psi_i)_{i\leq n}$ be a t-stratification reflecting $((\phi_i)_{i\leq n},\delta)$. Implicitly substituting $\phi _{\leq i}\vee \delta$ by $\bigvee _{j\leq i}((\phi_{\leq j})^\Delta\vee \delta_{j+1}^\Delta\vee \dots \vee \delta_n^\Delta)$,~\cite[Lemma 4.20]{immi} tells us that setting 
\begin{equation}
  \phi'_{\leq i}=
  \begin{cases}
    \psi _{\leq i}& \mbox{ if } i\leq d-1,\\
    \bigvee \limits_{j\leq i}((\phi_{\leq j})^\Delta\vee \delta_{j+1}^\Delta \dots \vee \delta_n^\Delta)\vee \psi_{\leq d-1}& \mbox{ if } i\geq d,
  \end{cases} \notag
\end{equation}
defines a t-stratification $(\phi_i')_{i\leq n}$ reflecting $(\phi_i)_{i\leq n}$. 

For the definition of the formulas $\phi_i^*$ we first fix a formula $(\psi _{\leq i})^\Delta \in \Delta$ for each $\psi_{\leq i}$ as in (ii). Define
\[
\phi_{\leq i}^*=
\begin{cases}
(\psi _{ \leq i})^\Delta& \mbox{ if } i\leq d-1,\\
\bigvee \limits_{j\leq i}((\phi_{\leq j})^\Delta\vee \delta_{j+1}^\Delta\vee \dots \vee \delta_n^\Delta)\vee (\psi _{\leq d-1})^\Delta & \mbox{ if } i\geq d.
\end{cases}
\]
Clearly $\phi _{\leq i}^*\in \Delta$ and $\phi _{\leq i}'\rightarrow \phi _{\leq i}^*$ for all $i\leq n$. Finally, we only need to check that for each $i\leq n$, $\dim(\phi^*_{\leq i}\wedge \neg (\phi '_{\leq i}))\leq d-1$. To prove this we show that the formula
\begin{equation}\label{eq:impl}
[\phi^*_{\leq i}\wedge \neg (\phi '_{\leq i})]\rightarrow [(\psi _{\leq d-1})^\Delta \wedge \neg(\psi _{\leq d-1})] 
\end{equation}
 holds for any $i\leq n$. In such case $(\psi_i)_{i\leq n}$ being a t-stratification and the properties of the formula $(\psi_{d-1})^\Delta$ imply the required inequality. By the definition of $\phi_{\leq i}^*$,~\eqref{eq:impl} is trivial for $i\leq d-1$, so assume that $i\geq d$. Hence,
\begin{flalign*}
 \phi^*_{\leq i}\wedge \neg (\phi '_{\leq i}) =& \left(\bigvee _{j\leq i} (\phi_{\leq j}^\Delta\vee \delta_j^\Delta\vee \dots \vee \delta _n^\Delta)\vee \psi_{\leq d-1}^\Delta\right) & \\
 &\wedge \neg \left(\bigvee _{j\leq i} (\phi_{\leq j}^\Delta\vee \delta_j^\Delta\vee \dots \vee \delta _n^\Delta)\vee \psi_{\leq d-1}\right) & \\
\longrightarrow & \ \psi _{\leq d-1}^\Delta\wedge \neg (\psi_{\leq d-1}).
\end{flalign*}
This ends the proof of the claim.
\end{proof}

We now aim to apply this proposition to the set $\Delta$ of all $L$-formulas. Clearly such set $\Delta$ satisfies Proposition~\ref{deltaset} (i), while (ii) is deduced after showing that any $\lhen$-definable set is contain in an $L$-definable set of the same dimension. For this we digress slightly to talk about cell decomposition. We show that what we require holds for cells, the conclusion then clearly holds for any $\lhen$-definable set. As usual, we let $(R,\rve)$ be a model of $\then$. We also work with the associated $\lrv$-structure $(R,\rvu)$. 

We define two kinds of cells, the first coming from the o-minimal setting of $R$ as an $L$-structure, and the second coming from the weakly o-minimality (on the first sort) of $(R,\rvu)$. 
\begin{Definition}[o-minimal cells] An $L$-1-cell is simply an interval of $R$ (allowing $\pm \infty$ as endpoints). An $L$-$(n+1)$-cell is a subset $C$ of $R^{n+1}$ for which there is an $L$-$n$-cell $C'\subseteq R^n$ such that either
    \begin{enumerate}[(i)] 
      \item $C=\{(a,f(a))\in R ^{n+1}\ |\ a\in C'\}$, where $f:R^n\longrightarrow R$ is an $L$-definable function, or
      \item $C=\{(a,b)\in R ^{n+1}\ |\ a\in C'\mbox{ and } f(a)<b<g(a)\}$, where $f,g:R^n\longrightarrow R$ are $L$-definable functions satisfying $f(x)<g(x)$ for all $x\in C'$.
    \end{enumerate}
We say that $X$ is a cell if it is an $n$-cell for some $n\geq 1$.
\end{Definition}
The idea for the second kind of cells was first developed in~\cite[Subsection 4.2]{dugald}, where a Cell Decomposition Theorem was proved for weakly o-minimal theories. Below $\widehat{R}$ denotes the Dedekind completion of $R$ as an ordered field.
 
\begin{Definition} An $\lrv$-1-cell is a convex $\lrv$-definable subset of $R$. An $\lrv$-$(n+1)$-cell is a subset $D$ of $R^{n+1}$ for which there is an $\lrv$-$n$-cell $D'\subseteq R^n$ such that either
    \begin{enumerate}[(i)] 
      \item $D=\{(a,f(a))\in R ^{n+1}\ |\ a\in D'\}$, where $f:R^n\longrightarrow R$ is an $\lrv$-definable function, or
      \item $D=\{(a,b)\in R ^{n+1}\ |\ a\in D'\mbox{ and } f(a)<b<g(a)\}$, where $f,g:R^n\longrightarrow \widehat{R} $ are definable functions satisfying $f(x)<g(x)$ for all $x\in D'$.
    \end{enumerate}
Similarly, we say that $X$ is an $\lrv$-cell if it is an $\lrv$-$n$-cell for some $n\geq 1$.
\end{Definition}

Notice that we do not expect the boundary functions to be continuous for any of the two kinds of cells defined above. In the case of $\lrv$-cells this omission is needed for the following theorem.

\begin{Theorem}
Let $X\subseteq R ^n$ be an $\lrv$-definable subset. Then it admits a partition into finitely many $\lrv$-cells. 
\end{Theorem}
\begin{proof}
 It follows from weakly o-minimality by \cite[Theorem 4.6]{dugald}. 
\end{proof}

The following is a corollary of this theorem and Proposition~\ref{sharperwomono}.
 \begin{Corollary}\label{cellpartition}
 Let $f:X\longrightarrow R$ be an $\lrv$-definable function with $X\subseteq R^n$. Then there are a finite partition of $X$ into finitely many $\lrv$-cells $D_1,\dots, D_m$, and $L$-definable functions $f_1,\dots,f_m:R^n\lto R$ such that $f|_{C_i}=f_i|_{C_i}$, for each $i\in \{1,\dots,m\}$. 
 \end{Corollary}

The following is a crucial result. Its proof makes clear why we did not insist on the continuity of the functions in the definition of ($L$-definable) cells.
\begin{Theorem}
If $D$ is an $\lrv$-cell, then there are finitely many $L$-cells $C_1,\dots ,$ $C_m$ such that $D\subseteq C_1\cup \dots \cup C_m$ and $\dim(D)=\dim(C_1\cup \dots \cup C_m)$.
\end{Theorem}
\begin{proof}
For an $\lrv$-1-cell this is obvious. Now suppose that $D'$ is an $\lrv$-$n$-cell and, as inductive hypothesis, assume that there are finitely many $L$-cells $C_i'$ such that $D'\subseteq \bigcup_i C'_i$ and $\dim(D')=\dim(\bigcup_i C'_i)$.\newline
\textbf{Case (i).} If $X=\{(a,f(a))\in R ^{n+1}\ |\ a\in D'\}$, where $f:R^n\longrightarrow R$ is an $\lrv$-definable function. By applying Corollary~\ref{cellpartition}, let $\{A_j\}_j$ be a decomposition of $\bigcup _iD'_i$ into cells and $\{f_j:R ^n\longrightarrow R\}_j$ $L$-definable functions such that  $f|_{A_j}=f_j|_{A_j}$, for each $j$. Define for all $i$ and $j$,
$C_{i,j}:=\{(a,f_j(a))\ |\ a\in D'_i\}$. Then clearly each $C_{i,j}$ is an $L$-cell and $X\subseteq \bigcup_{i,j}C_{i,j}$. Moreover, notice that for any $i$ and $j$, $\dim(C_{i,j})=\dim(D'_i)$; hence, by properties of dimension,  
\[\dim \Big( \bigcup _{i,j}\nolimits C_{i,j}\Big)=\dim \Big(\bigcup_i\nolimits D'_i\big)=\dim(D')=\dim(X),\] 
as required.\newline
\textbf{Case (ii).} If $X=\{(a,b)\in R ^{n+1}\ |\ a\in D'\mbox{ and } f(a)<b<g(a)\}$, where $f,g:R^n\longrightarrow \widehat{R}$ are definable functions satisfying $f(x)<g(x)$ for all $x\in D'$. Similarly as in the previous case, let $\{A_j\}_j$ be a decomposition of $\bigcup_i D'_i$ into cells and $\{f_j\}_j,\{g_j\}_j$ be L-definable functions such that $f|_{A_j}=f_j|_{A_j}$, $g|_{A_j}=g_j|_{A_j}$  and $f_j|_{A_j}<g_j|_{A_j}$, for each $j$. Define for all $i$ and $j$ $C_{i,j}=\{(a,b)\ |\ a\in D'_i\mbox{ and } f_j(a)<b<g_j(a)\}$.
Clearly again each $C_{i,j}$ is an $L$-cell and $X\subseteq \bigcup_{i,j}C_{i,j}$. Also, for any $i,j$, $\dim(C_{i,j})=\dim(D'_i)+1$; this implies that 
\[\dim \Big(\bigcup_{i,j}\nolimits C_{i,j}\Big)=\dim \Big(\bigcup_i\nolimits D'_i\Big)+1=\dim (D')+1=\dim(X),\]
finishing this case.
\end{proof}

It follows that any $\lrv$-definable subset of $R^n$ is contained in an $L$-definable set of the same dimension. The same holds for an $\lhen$-definable subset of $R^n$, for it is $\lrv$-definable after picking some parameters in the sort $\rvu$.

\begin{Corollary}
 Let $x$ be a tuple of field sort variables and $\phi(x)$ and $\lhen$-formula. Then there is an $L$-formula $\psi(x)$ such that $\phi(x)\lto \psi(x)$ and $\dim(\phi(x))=\dim(\psi(x))$.
\end{Corollary}

From this corollary we conclude that the set of all $L$-definable formulas satisfies the conditions (i) and (ii) in Proposition~\ref{deltaset}. We have thus proved the following. 

\begin{Theorem}\label{ldefstrats}
 If $(S_i)_{i\leq n}$ is a t-stratification \textup{(}of $B_0$\textup{)}, then there is an $L$-definable t-stratification $(S_i')_{i\leq n}$ reflecting $(S_i)_{i\leq n}$ (and thus reflecting any other map reflected by $(S_i)_{i\leq n}$).
\end{Theorem}

\subsection{Application 2: Archimedean t-stratifications}
Here we look at the stratifications induced on the real field by t-stratifications in a non-standard model, pointing out that the results in~\cite[Section 7]{immi} hold in our context. 

Recall that $L$ expands the language $\lori$ and that $T$ is a power bounded o-minimal $L$-theory expanding the theory of real closed fields. We regard the real field $\mathbb R$ as an $L$-structure modelling $T$. We let $\mathcal U$ be a non-principal ultrafilter on $\mathbb N$ and ${}^*\mathbb R$ the ultrapower $\mathbb R^{\mathbb N}/\mathcal U$. Then $\starr$ is a non-archimedean model of $T$ and the set of limited numbers $O:=\{x\in \starr\ |\ \exists n\in \mathbb N(-n\leq x\leq n)\}$ is a $T$-convex subring of $\starr$; this is true because the convex hull of an elementary substructure is always a $T$-convex subring---see~\cite[(2.7)]{driesI}---and $O$ equals the convex hull of $\mathbb R$ in $\starr$. The image of $x\in \mathbb R$ under the elementary embedding of $\mathbb R$ into $\starr$ is denoted by $^*x$. Accordingly, if $X\subseteq \mathbb R^n$, $^*X$ denotes its non-standard version in $\starr$ (i.e. $^*X:=X^\mathbb N/\mathcal U$). 

\begin{Definition}
	Let $(A_i)_{i\leq n}$ be a partition of $\mathbb R^n$ and $X\subseteq \mathbb R^n$. We say that $(A_i)_{i\leq n}$ is an \emph{archimedean t-stratification} of $X$ if each $A_i$ is $L$-definable and $({}^*A_i)_{i\leq n}$ is a t-stratification of $^*X$. 
\end{Definition}

\begin{Remark}\label{archtstratsexist} By Theorem~\ref{ldefstrats} any $L$-definable subset of $\mathbb R^n$ admits an archimedean t-stratification.	
\end{Remark}

 The main theorem in this subsection is the following.

\begin{Theorem}\label{archimedeantstrat}
	If $(A_i)_{i\leq n}$ is an archimedean t-stratification of $X\subseteq \mathbb R^n$, then $(A_i)_{i\leq n}$ is a $C^1$-Whitney stratification of $X$.
\end{Theorem}
\begin{proof} For a definition of Whitney stratifications see~\cite[Section 9.7]{rag} or---more in accordance with our interests---\cite[Definition 7.8]{immi}. We limit ourselves to justify that the proof of~\cite[Proposition 7.10]{immi} can be carried out in our context.

By~\cite[Proposition 5.8]{driesI} the value group $\Gamma$ and the residue field $\overline R$ are \emph{orthogonal}\footnote{i.e. any definable set $X\subseteq \Gamma^l\times \overline R^m$ is a finite union of of sets of the form $E\times F$ where $E\subseteq \Gamma^l$ and $F\subseteq \overline R^m$ are definable.}; parenthetically, this is true only when $T$ is power bounded. Hence,~\cite[Hypothesis 7.1]{immi} holds in our context, and the conclusion of~\cite[Corollary 7.6]{immi} does too in consequence. The proof of~\cite[Proposition 7.10]{immi} can be then carried out with minor, straightforward modifications (e.g. `semi-algebraic' needs to be substituted by `$L$-definable' everywhere). That proposition claims that any archimedean t-stratification of the whole of $\mathbb R^n$ is a Whitney stratification. The implicit assumption that $({}^*A_i)_{i\leq n}$ is t-stratification of $^*X$ easily implies that the Whitney stratification $(A_i)_{i\leq n}$ is in fact one of $X$, see for example the second part of the proof of~\cite[Theorem 7.11]{immi}.
\end{proof}

The next is an easy result that follows from Remark~\ref{archtstratsexist} and the last theorem.

\begin{Corollary}
   Every $L$-definable subset of $\mathbb R^n$ admits a Whitney stratification.
\end{Corollary}

This corollary is superseded by a result of T. L. Loi (\cite{loi}) showing that Whitney stratifications exists for all definable sets in \emph{any} o-minimal expansion of the real field. The corollary above is limited to such expansion being power bounded. 

Our last result exhibits a property of archimedean t-stratifications that Whitney stratifications do not posses. For $X\subseteq \mathbb R^n$ and $p\in \mathbb R^n$, the \emph{tangent cone of $X$ at $p$} is the set 
\[
C_p(X):=\{y\in \mathbb R^n\ |\ \forall \varepsilon\in \mathbb R_{>0}\exists x\in X,r\in \mathbb R_{>0}(\|x-p\|<\varepsilon\wedge \|r(x-p)-y\|<\varepsilon)\};
\]
where $\|\cdot\|$ is the usual norm on $\mathbb R^n$. 

Since the \emph{Euclidean topology} on $^*\mathbb R^n$---that induced by $\|\cdot\|$ and in fact given by $\sum_1^nx_i^2$ for any $(x_1,\dots ,x_n)\in {}^*\mathbb R^n$---coincides with the valuative topology on $^*\mathbb R^n$, we can see that ${}^*(C_p(X))=\mathcal C_{^*p}({}^*X)$ whenever $X\subseteq \mathbb R^n$ is $L$-definable and $p\in \mathbb R^n$. This fact and Theorem~\ref{intro1} imply our last result.

\begin{Corollary}
  Let $X$ be an $L$-definable subset of $\mathbb R^n$ and $p$ a point in $\mathbb R^n$. Given a partition $(A_i)_{i\leq n}$ of $\mathbb R^n$, define the sets $C_{p,0}:=C_p(A_0)$ and $C_{p,i+1}:=C_{p}(A_0\cup \cdots \cup A_{i+1})\setminus C_{p}(A_0\cup \cdots \cup A_{i})$ for $0\leq i< n$. If $(A_i)_{i\leq n}$ is an archimedean t-stratification of $X$, then $(C_{i,p})_{i\leq n}$ is an archimedean t-stratification of $C_p(X)$.
\end{Corollary}
From Theorem~\ref{archimedeantstrat} we see that an archimedean t-stratification of $X$ \emph{induces} a Whitney stratification on the tangent cones of $X$; the analogous claim is false if we start instead  with a Whitney stratification (see~\cite[Example 3.16]{egr}).

\addcontentsline{toc}{chapter}{Bibliography}
\bibliographystyle{plain}
\bibliography{Thesisbi}

\end{document}